\documentclass[11pt]{amsart}

\usepackage{amsfonts, amssymb, amsmath, amsthm, enumerate,amsrefs,appendix,todo}

\theoremstyle{definition}
\newtheorem{definition}{Definition}[section]
\newtheorem{remark}[definition]{Remark}

\theoremstyle{plain}
\newtheorem{theorem}[definition]{Theorem}
\newtheorem{thm}[definition]{Theorem}
\newtheorem{proposition}[definition]{Proposition}
\newtheorem{prop}[definition]{Proposition}
\newtheorem{lemma}[definition]{Lemma}
\newtheorem{corollary}[definition]{Corollary}

\newtheorem*{claim}{Claim}

\newcommand{\from}{\colon}

\newcommand{\define}[1]{\emph{#1}}

\newcommand{\N}{\mathbb{N}}
\newcommand{\Z}{\mathbb{Z}}

\newcommand{\T}{\mathbb{T}}
\newcommand{\E}{\mathrel{E}}
\newcommand{\G}{\mathrel{G}}

\newcommand{\EG}{{\mathrel{E_G}}} 
\newcommand{\evenG}{\mathrel{E_{2,G}}}
\newcommand{\id}{\mathrm{id}}

\renewcommand{\subset}{\subseteq}
\renewcommand{\supset}{\supseteq}
\newcommand{\inters}{\cap}
\newcommand{\biginters}{\bigcap}
\newcommand{\union}{\cup}

\newcommand{\bigunion}{\bigcup}
\newcommand{\restrict}{\restriction}

\newcommand{\bfSigma}{\mathbf{\Sigma}}

\newcommand{\bigdisjointunion}{\bigsqcup}
\newcommand{\mathand}{\mbox{ and }}
\newcommand{\mathor}{\mbox{ or }}

\DeclareMathOperator{\Col}{Col}
\DeclareMathOperator{\Cay}{Cay}
\DeclareMathOperator{\graph}{graph}
\DeclareMathOperator{\wmsf}{WMSF}
\DeclareMathOperator{\freeactions}{FR}
\DeclareMathOperator{\dom}{dom}
\renewcommand{\deg}{\mathrm{deg}}

\DeclareMathOperator{\Aut}{Aut}

\begin{document}

\title{Brooks's theorem for measurable colorings}

\author{Clinton T.~Conley}
\address[Clinton T.~Conley]{Carnegie Mellon University}
\email{clintonc@andrew.cmu.edu}

\author{Andrew S.~Marks}
\address[Andrew S.~Marks]{University of California, Los Angeles}
\email{marks@math.ucla.edu}

\author{Robin D.~Tucker-Drob}
\address[Robin D.~Tucker-Drob]{Texas A\&M University}
\email{rtuckerd@math.tamu.edu}

\thanks{The first author is supported by NSF grant DMS-1500906 and the
second author is supported by NSF grant DMS-1500974 and the John Templeton
Foundation under Award No. 15619.}

\date{\today}

\begin{abstract}
We generalize Brooks's theorem to show that if
$G$ is a Borel graph on a standard Borel space $X$ of degree
bounded by $d \geq 3$ which contains no $(d+1)$-cliques, 
then $G$ admits a $\mu$-measurable $d$-coloring with respect to any Borel
probability measure $\mu$ on $X$, and a Baire measurable $d$-coloring with
respect to any compatible Polish topology on $X$. The proof of this theorem 
uses a new technique for constructing one-ended spanning subforests of
Borel graphs, as well as ideas from the study of list colorings.
We apply the
theorem to graphs arising from group actions to obtain factor of IID $d$-colorings of
Cayley graphs of degree $d$, except in two exceptional cases.
\end{abstract}

\maketitle

\section{Introduction}

We
begin by recalling a classical theorem of Brooks from finite combinatorics.

\begin{thm}[Brooks's Theorem \cite{D}*{Theorem 5.2.4}]\label{thm:ClassicBrooks}
Suppose $G$ is a finite graph with vertex degree
bounded by $d$. Suppose further that $G$ contains no complete graph
on $d + 1$ vertices, and if $d = 2$ that $G$ contains no odd
cycles. Then $G$ has a (proper) vertex $d$-coloring.
\end{thm}

It is easy to see that if a graph $G$ has vertex degree bounded by
$d$, then $G$ has a $(d + 1)$ coloring: greedily color the vertices
one by one, using the least color not already assigned to a neighboring
vertex. One way of regarding Brooks's theorem is that it is a complete
characterization of the graphs for which this obvious upper bound cannot be
improved: odd cycles and complete graphs.

Brooks's theorem is a fundamental result of graph coloring which has been
generalized in a variety of different
settings. See \cite{CR} for a recent survey. This paper examines measurable
generalizations of Brooks's Theorem. While a straightforward compactness
argument extends Brooks's Theorem to infinite graphs, such an argument
cannot in general produce a coloring with desirable measurability properties such as
being $\mu$-measurable with respect to some probability measure, or
being Baire measurable with respect to some \define{Polish} (separable, completely
metrizable) topology. Recall
that a \define{standard Borel space} is a set $X$ equipped with a
$\sigma$-algebra generated by a Polish topology. The graphs we will consider are \define{Borel graphs}
where the vertices of the graph are the
elements of some standard Borel space $X$, and whose edge relation is Borel
as a subset of $X \times X$.

In studying Borel graphs, Kechris, Solecki, and
Todorcevic~\cite{KST}*{Proposition 4.6} have shown that every Borel graph
of vertex degree bounded by a finite $d$ admits a Borel $(d + 1)$-coloring,
in analogy to the fact above. Hence, we can ask again for which
exceptional cases this obvious bound can not be improved.
Marks~\cite{Ma}*{Theorem 1.3} has shown that for every finite $d$, there
is an acyclic Borel graph of degree $d$ with no Borel $d$-coloring. Hence,
to obtain a reasonable measurable analog of Brooks's theorem as in
Theorem~\ref{thm:mb}, we must consider measurability constraints weaker
than Borel measurability. In this paper we focus on Baire measurability with respect to some compatible Polish topology and
$\mu$-measurability with respect to some Borel probability measure $\mu$.

Still, the obvious analogue of the
$d=2$ case of Brooks's Theorem does not hold for either of these
measurability notions. Let $S : \T \rightarrow \T$ be an irrational
rotation of the unit circle $\T$, and let $G_S$ be the graph on $\T$
rendering adjacent each point $x \in \T$ and its image $S(x)$ under $S$
so $G_S$ is acyclic and each vertex has degree $2$. Now an easy ergodicity
argument shows that $G_S$ has no Lebesgue measurable $2$-coloring: since
$S$ is measure preserving, the two color sets would have to have equal measure,
but since $S^2$ is ergodic, the color sets would each have to be null or
conull. Similarly, $G_S$ has no Baire measurable $2$-coloring (see
Section~\ref{sec:d=2}).  This example may be considered an infinite analog of the odd cycle exemption in \ref{thm:ClassicBrooks}.

Our main result is the following
measurable analogue of Brooks's theorem for the case $d \geq 3$.
\begin{thm}\label{thm:mb}
Suppose that $G$ is a Borel graph on a standard Borel space $X$ with vertex
degree bounded by a finite $d \geq 3$. Suppose further that $G$ contains
no complete graph on $d+1$ vertices.
\begin{enumerate}
\item Let $\mu$ be any Borel probability measure on $X$. Then $G$ admits a $\mu$-measurable $d$-coloring.
\item Let $\tau$ be any Polish topology compatible with the Borel structure on $X$. Then $G$ admits a Baire measurable $d$-coloring.
\end{enumerate}
\end{thm}
This improves a prior result of Conley and Kechris who proved an
analogous theorem for approximate colorings where one is allowed to discard
a set of arbitrarily small measure~\cite{CK}*{Theorems 2.19, 2.20}.

Our proof of Theorem~\ref{thm:mb} uses ideas from the study of list
colorings in graph theory. Recall that if $G$ is a graph on $X$ and $L$ is
a function mapping each $x \in X$ to a set $L(x)$, then a \define{coloring
of $G$ from the lists $L$} is a coloring $c$ of $X$ such that for every
$x$, $c(x) \in L(x)$. Given a function $f \from X \to \N$, we say $X$ is
$f$-list-colorable if for every function $L$ on $X$ with $|L(x)| = f(x)$,
there is a coloring of $G$ from the lists $L$. We say $G$ is
degree-list-colorable if it is $f$-list-colorable for the function $f(x) =
\deg_G(x)$, where $\deg_G(x)$ is the degree of $x$. Borodin~\cite{B} and
Erd\H{o}s-Rubin-Taylor~\cite{ERT} independently generalized Brooks's
theorem to list colorings by classifying the finite graphs $G$ which are
not degree-list-colorable. Of course, if a graph $G$ has degree bounded by
$d$, then being degree-list-colorable implies that $G$ is $d$-colorable
since we can color from the lists $L(x) = \{1,\ldots, \deg_G(x)\}$ for every
$x$. For degree-list-coloring, the exceptional graphs are the finite
Gallai trees:
\begin{thm}[Borodin~\cite{B}, Erd\H{o}s-Rubin-Taylor~\cite{ERT}] 
A finite connected graph is degree-list-colorable iff it is not a Gallai tree.
\end{thm}

Recall here that a set $S$ of vertices from a graph $G$ is 
\define{biconnected} if the induced subgraph $G \restriction S$ remains
connected after removing any single vertex from $S$. A \define{block} of a
graph is a maximal biconnected set and a \define{Gallai tree} is a
connected graph whose blocks are complete graphs or odd cycles. 


In addition to making key use of this result, we also generalize it to Borel
graph colorings. Recall that if $Y$ is a set, we use $[Y]^{< \infty}$ to
denote the collection of finite subsets of $Y$, and if $Y$ is a standard Borel space then
so is $[Y]^{< \infty}$ with the Borel structure induced as a quotient of $\bigdisjointunion_{n \in \N} Y^n$.
Say that a locally finite Borel graph $G$ on $X$ is
\define{Borel degree-list-colorable} if for every Polish space and every
Borel function $L \from X \to [Y]^{< \infty}$ such that $|L(x)| = \deg_G(x)$,
there is a Borel coloring  $c \from X \to Y$ of $G$ from the lists $L$.
\begin{thm}\label{thm:Borel_Gallai}
Suppose that $G$ is a locally finite Borel graph on a standard Borel space
$X$, and that no connected component of $G$ is a Gallai tree. Then $G$ is Borel degree-list-colorable.
\end{thm}

Recall that if $G$ is a locally finite graph on $X$, then two rays $(x_i)_{i \in \N}$ and
$(y_i)_{i \in \N}$
in $G$ are \define{end-equivalent} if for every finite set $S \subset X$,
the rays eventually lie in the same connected component of $G \restriction
(X \setminus S)$. If $G$ is an acyclic graph, this is equivalent to $(x_i)$
and $(y_i)$ being tail equivalent. An \define{end} of a graph is an
end-equivalence class of rays. A \define{one-ended spanning subforest} of
$G$ is a acyclic graph $T \subset G$ on $X$ such that every $x \in X$
is incident on some edge in $T$, and every connected component of $T$ has
exactly one end.

The other tool we use to prove Theorem~\ref{thm:mb} is a new technique for
constructing $\mu$-measurable and Baire measurable one-ended spanning
subforests of acyclic Borel graphs. In the case where the connected
components of $G$ are Gallai trees and we cannot apply
Theorem~\ref{thm:Borel_Gallai}, we use a one-ended subforest of the Gallai
tree to give a skeleton along which we may color the graph to prove
Theorem~\ref{thm:mb}.

\begin{theorem}
\label{thm:general_one_end}
  Suppose that $G$ is a locally finite acyclic Borel graph on a standard
  Borel space $X$ such that no connected component of $G$ has $0$ or $2$
  ends.
  \begin{enumerate}
   \item Let $\mu$ be a Borel probability measure on $X$. Then there is a
   $\mu$-conull Borel set $B$ and a one-ended Borel function $f \from B \to
   X$ whose graph is contained in $G$.

   \item Let $\tau$ be a compatible Polish topology on $X$. Then there is a
   $\tau$-comeager Borel set $B$ and a one-ended Borel function $f \from B
   \to X$ whose graph is contained in $G$.
   \end{enumerate}
\end{theorem}

We also discuss a version of Theorem~\ref{thm:general_one_end}.1 for locally countable
graphs in Section~\ref{section:subforests}.

In Section \ref{sec:applications} we
apply Theorem \ref{thm:mb} to graphs arising from group actions, and we
apply the methods of Section \ref{section:subforests} along with results
from probability to obtain factor of IID $d$-colorings of Cayley graphs of
degree $d$, apart from two exceptional cases.

Finally, in the case $d = 2$ we show that the ergodic-theoretic obstruction
discussed above is in a sense the only counterexample to Brooks's theorem.

\begin{thm}\label{thm:mb2}
  Suppose $G$ is a Borel graph on a standard Borel space $X$ with vertex
  degree bounded by $d = 2$ such that $G$ contains no odd cycles. Let
  $\evenG$ be the equivalence relation on $X$ where $x \evenG y$ if $x$ and
  $y$ are connected by a path of even length in $G$.
  \begin{enumerate}
  \item Let $\mu$ be a $G$-quasi-invariant Borel probability measure on
  $X$. Then $G$ admits a $\mu$-measurable $2$-coloring if and only if there does not
  exist a non-null $G$-invariant Borel set $A$ such that every
  $\evenG$-invariant Borel subset of $A$ differs from a $G$-invariant set by a
  nullset.

  \item Let $\tau$ be a $G$-quasi-invariant Polish topology
  compatible with the Borel structure on $X$. Then $G$ admits a
  Baire measurable $2$-coloring if and only if there does not exist a non-meager
  $G$-invariant Borel set $A$ so that every $\evenG$-invariant Borel subset
  of $A$ differs from a $G$-invariant set by a meager set.
  \end{enumerate}
\end{thm}

\section{Preliminaries}\label{sec:prelims}

A (simple, undirected) \define{graph} on a set of vertices $X$ is a
symmetric irreflexive relation on $X$. Given a graph $G$ on $X$, we say that two
points $x,y \in X$ are \define{neighbors} or are \define{adjacent} in $G$
if $x \mathrel{G} y$. The ($G$-)\define{degree} of a vertex $x$, denoted
$\deg_G(x)$, is the number of neighbors of $x$
and a graph $G$ has bounded degree $d$ if every vertex has
degree at most $d$. We say that $G$ is \define{locally finite} (resp.\ \define{locally countable}) if the degree
of every vertex in $G$ is finite (resp.\ countable). A set $A \subset X$
of vertices of $G$ is ($G$-)\define{independent} if for every $x,y \in A$
it is not the case that $x \mathrel{G} y$. If $f \from X \to X$ is a
function, then the graph $G_f$ generated by $f$ is defined by $x
\mathrel{G_f} y$ if $x \neq y$ and either $f(x) = y$ or $f(y) = x$.

A (simple) \define{path} in a graph $G$ is a finite sequence $x_0, \ldots,
x_n$ of distinct vertices such that $x_0
\mathrel{G} x_1 \mathrel{G} \ldots \mathrel{G} x_n$. We say that such a path has
\define{length}
$n$. A
(simple)
\define{ray} is an infinite sequence $(x_i)_{i \in \N}$ of distinct vertices such that $x_i
\mathrel{G} x_{i+1}$ for every $i \in \N$, and a \define{line} is a bi-infinite
sequence $(x_i)_{i \in \Z}$ such that $x_i \mathrel{G} x_{i+1}$ for every
$i \in \Z$.
If $G$ is a graph on $X$,
then the graph metric $d_G: X^2 \to \N \union \{\infty\}$ on $G$ maps $(x, y)
\in X^2$ to the length of the shortest path connecting $x$ and $y$, if such a
path exists; otherwise we set $d_G(x,y) = \infty$. A \define{cycle} in a graph $G$ is a sequence of vertices $x_0
\mathrel{G} x_1 \mathrel{G} x_2 \mathrel{G} \ldots \mathrel{G} x_n = x_0$ such that $n
> 2$, and $x_i \neq x_j$ for all $i< j < n$. We say the length of such a
cycle is $n$. We say that a graph is
\define{acyclic} if it does not contain any cycles. If $G$ is a graph on
$X$, then we let $E_G$ be the connectedness relation of $G$, where $x
\mathrel{E_G} y$ if there is a path in $G$ from $x$ to $y$. We say that a
set $A \subset X$ is \define{$G$-invariant} if it is $E_G$-invariant, that is, a
union of connected components of $G$.

A \define{Borel graph} is a graph whose vertices are the elements of a
standard Borel space $X$, and whose edge relation is Borel as a subset of
$X \times X$. The restriction $G \restrict A$ of $G$ to a set $A \subset X$
is the graph on $A$ is the induced subgraph obtained by restricting the relation $G$ to $A$. If $G$
is a Borel graph, and $A$ is a Borel set, then since $A$ inherits the
standard Borel structure of $X$, we see that $G \restrict A$ is also a
Borel graph.

Let $X$ and $Y$ be standard Borel spaces. Let $\mu$ be a Borel probability
measure on $X$. We say that a function $f: X\rightarrow Y$ is
\define{$\mu$-measurable} if it is measurable for the completion of $\mu$.
Let $\tau$ be a compatible Polish topology on $X$ (by \define{compatible}
we mean that the $\sigma$-algebra generated by the $\tau$-open sets coincides
with the given Borel $\sigma$-algebra on $X$). We say that a function
$f:X\rightarrow Y$ is \define{Baire measurable (with respect to $\tau$)} if
it is measurable for the $\sigma$-algebra of sets which have the Baire
property with respect to the completion of $\tau$; the smallest $\sigma$-algebra containing the Borel sets and all $\tau$-meager sets.

There is an equivalence between the admitting a
$\mu$-measurable or Baire measurable coloring, and admitting a Borel
coloring modulo an invariant null or meager set:

\begin{prop}\label{prop:measurable_equiv}
  Suppose $G$ is a locally countable Borel graph on a standard Borel space
  $X$. Suppose that $G$ admits some $n$-coloring.
  \begin{enumerate}
    \item Let $\mu$ be any Borel probability measure on $X$. Then $G$
    admits a $\mu$-measurable $n$-coloring if and only if there is a
    $\mu$-conull
    $G$-invariant Borel set $A \subset X$ such that
    $G \restrict A$ has a Borel $n$-coloring.

    \item Let $\tau$ be any Polish topology compatible with the Borel
    structure on $X$. Then $G$
    admits a Baire measurable $n$-coloring if and only if there is a
    comeager
    $G$-invariant Borel set $A \subset X$ such that
    $G \restrict A$ has a Borel $n$-coloring.
  \end{enumerate}
\end{prop}
\begin{proof}
We begin with the direction $\Rightarrow$ of \emph{1}. Suppose $c$ is a
$\mu$-measurable $n$-coloring of $G$. By the Feldman-Moore theorem
\cite{KM}*{Theorem 1.3}, let
$\{T_i\}_{i \in \N}$ be a set of Borel automorphisms of $X$ so that
$\EG = \bigunion_{i \in \N} T_i$.
Now for each $i$, since $c \circ T_i$ is $\mu$-measurable, there is a
$\mu$-conull  set $A_i$ such that $c \circ T_i \restrict A_i$ is Borel.
Thus, if $A = [\biginters_{i \in \N}A_i]_{E_G}$, then $A$ is $\mu$-conull, and $c
\restrict A$ is Borel. This is because for all $x \in A$, if $i$ is least
such that $T_i^{-1}(x) \in A_i$, then $c(x) = (c \circ T_i)(T_i^{-1}(x))$.

The direction $\Leftarrow$ of \emph{1} is straightforward. Given a Borel
$n$-coloring $c$ of $G \restrict A$ where $A$ is a $\mu$-conull
$G$-invariant set, let $c'$ be an arbitrary $n$-coloring of $G
\restrict (X \setminus A)$. Then $c \union c'$ is a $\mu$-measurable
coloring of $G$.

The proof of part \emph{2} is identical to the above. Simply replace the phrase
$\mu$-conull with comeager (with respect to $\tau$), and $\mu$-measurable
with Baire measurable (with respect to $\tau$).
\end{proof}

While we have stated our main results in terms of the
existence of $\mu$-measurable and Baire measurable colorings, throughout
the paper we will
mostly work with the equivalent formulations given by
Proposition~\ref{prop:measurable_equiv} above.
Note here that the classical Brooks's theorem shows the existence of the
requisite $d$-coloring that we will need to apply the above proposition.
An analogous fact is also true for list-coloring.

Suppose $G$ is a locally countable Borel graph on a standard Borel space $X$, and $\mu$ is a Borel probability
measure on $X$. Then we say that $\mu$ is
\define{$G$-quasi-invariant} if every $\mu$-null set is contained in a
$G$-invariant $\mu$-null set. Now
for every Borel probability measure $\mu$ on
$X$, there exists a $G$-quasi-invariant Borel probability measure $\mu'$ on
$X$ such every $\mu'$-null set
is $\mu$-null
(that is, $\mu'$ dominates $\mu$).
This follows from the Feldman-Moore theorem \cite{KM}*{Theorem 1.3}
by letting $\{T_i\}_{i \in \N}$ be a set of Borel automorphisms of $X$ such
that $\EG = \bigunion_i \graph(T_i)$, and then setting $\mu'(A) = \sum_{i
\geq 1} 2^{-n} \nu(T_i(A))$ (see~\cite{KM}*{Section 8}).
A key property of a quasi-invariant measure is that if $A$ is
$\mu'$-conull, then it contains a $G$-invariant $\mu'$-conull set. This is
because the set $\{x : x \notin A \land x \in [A]_{E_G}\}$ is null since it is
contained in the complement of $A$, and hence is saturation is null.

Similarly, suppose $G$ is a Borel graph on $X$, and $\tau$ is a compatible
Polish topology for $X$. Then we say that $\tau$ is
\define{$G$-quasi-invariant} if every $\tau$-meager set is contained in a
$G$-invariant $\tau$-meager set. It follows from a result of
Zakrzewski~\cite{Z} that
if $G$ is a locally countable Borel graph, then for every compatible Polish
topology $\tau$ on $X$ there is a $G$-quasi-invariant compatible Polish
topology $\tau'$ such that every $\tau'$-meager set is $\tau$-meager.

The combination of the above discussion and
Proposition~\ref{prop:measurable_equiv} justifies our assumption from now
on that our measures and topologies are quasi-invariant with respect to the
graphs we consider. This is because Proposition~\ref{prop:measurable_equiv}
allows us to reformulate Theorems~\ref{thm:mb} and \ref{thm:mb2} to state
the existence of a Borel $d$-coloring of $G \restrict A$ for some
$G$-invariant Borel $A$ which is conull or comeager. Thus, the assumption
of quasi-invariance is harmless since we may always pass to a
quasi-invariant measure or topology without adding any new conull or
comeager sets. Our assumption of quasi-invariance is helpful because it
frees us from talking constantly about invariant sets when we discard null
or meager set, since in this case a null set or meager set of vertices is always contained
in a null set or meager $G$-invariant set respectively.

\section{One-ended subforests}\label{section:subforests}

This section focuses on definably isolating one-ended subforests of various
classes of locally finite acyclic graphs. These subforests will subsequently provide a
skeleton along which to construct a coloring. A few of our results in this
section generalize to locally countable graphs and in these cases we
indicate what changes need to be made in our arguments to work in this
greater generality.  It is worth noting that since the graphs under investigation are
acyclic, the two natural definitions of end-equivalence in terms of deleting vertices
and deleting edges in fact coincide.

Suppose $G$ is an acyclic graph on $X$ and every connected component of $G$
has one end. We can form a function $f \from X \to X$ as follows. For each
$x \in X$, there is a unique infinite $G$-ray $(x_i)_{i \in \N}$ such that
$x = x_0$. Then
define $f(x) = x_1$, so $f$ points ``towards'' the unique end in the graph.
Note that $f$ generates $G$, and since $G$ is everywhere one-ended, there
is no infinite \define{descending} sequence $x_0, x_1, \ldots \in X$ with
$f(x_{i+1}) = x_i$ for every $i$. That is, as a relation, $f$ is well-founded.
Conversely, suppose $f \from X \to X$ is a function containing no infinite
descending sequence (in particular, no fixed points). Then the graph $G_f$ generated by $f$ has one end in every
connected component. Thus, if $H$ is an graph on $X$, finding a one-ended
spanning subforest of $H$ is equivalent to finding a function
$f \from X \to X$ contained in $H$ admitting no infinite descending sequences.

Now suppose $f \from X \to X$ is a partial function. We say that $f$ is
\define{one-ended} if there is no infinite descending sequence $x_0, x_1,
\ldots \in X$ in $f$. Note that if $f$ is a one-ended partial function,
then any connected component of $G_f$ containing a point $y$ not in the
domain of $f$ will contain $0$ ends (and not one end!). Our terminology
here is inspired by regarding such a $y \notin \dom(f)$ as a ``point at
infinity''. In the case where $f$ is finite-to-one, by K\"onig's lemma, $f$
is one-ended if and only if all of its backward orbits $f^{-\N}(x) =
\bigunion_{n \in \N} f^{-n}(x)$ are finite.

\begin{proposition}\label{prop:pointaway}
  Suppose that $G$ is a locally finite Borel graph on a standard Borel
  space $X$, and $A \subset X$ is Borel.  Then there is a one-ended Borel
  function $f \from [A]_\EG \setminus A \to [A]_\EG$ whose graph is
  contained in $G$.
\end{proposition}
\begin{proof}
  Without loss of generality we may assume that $[A]_\EG = X$ since $G
  \restriction [A]_{\EG}$ is a Borel graph. Let $B$ be
  the set of $x \in X\setminus A$ such that there exists a
  $G$-ray $(x_i)_{i \in \N}$ with $x_0 = x$ and $d(x_{i+1},A) > d(x_i,A)$
  for all $i \in \N$. Note that $B$ is Borel by K\"{o}nig's lemma.
  By Lusin-Novikov uniformization~\cite{K}*{Theorem 18.10}, there is a Borel function $f \from X
  \setminus A \to X$ such that
  \begin{enumerate}
    \item if $x \in B$, then $f(x)$ is a neighbor of $x$ such that $f(x)
    \in B$ and $d(f(x),A) > d(x,A)$.
    \item if $x \notin B$, then $f(x)$ is a neighbor of $x$ such
    that $d(f(x),A) < d(x,A)$.
  \end{enumerate}
  To see that $f$ is as desired, suppose first that $x \notin B$.  Then
  $f^{-\N}(x) \subset X \setminus B$, and if $f^{-\N}(x)$ were infinite an
  application of K\"{o}nig's lemma would allow the construction of an
  injective $G$-ray as in the definition of $B$, contradicting the fact
  that $x \notin B$.  On the other hand, if $x \in B$ then $f^{-\N}(x)
  \inters B$ is finite and is in fact contained in $\bigunion_{i < d(x,A)}
  f^{-i}(x)$.  Consequently $f^{-\N}(x)$ is the union of this finite set
  with $\bigunion_{i \leq d(x,A)}\{f^{-\N}(y) : y \in f^{-i}(x) \setminus
  B\}$, which by the previous case is a finite union of finite sets.
\end{proof}

We note that in the case where $G$ is a locally countable Borel graph, the
same function $f$ is one-ended. However, the function $f$ will not be Borel
in general since while $B$ is analytic it may not be Borel. Nevertheless,
$f$ will be $\sigma(\bfSigma^1_1)$-measurable, where $\sigma(\bfSigma^1_1)$
is the $\sigma$-algebra generated by the analytic sets. Hence, the
function $f$ is Borel after discarding an appropriate null or meager set.

Iteratively applying Proposition \ref{prop:pointaway} can be used to find
one-ended Borel functions whose domains are $G$-invariant.

\begin{lemma}\label{lem:iterative_oes}
  Suppose $G$ is a locally finite Borel graph on $X$, and there is a
  decreasing sequence $A_0 \supset A_1 \supset \ldots$ of Borel sets
  with empty intersection such that $A_{i+1}$ meets each
  connected component of $G \restriction A_i$. Then there is a one-ended
  Borel function $f \from [A_0]_{E_G} \to [A_0]_{E_G}$ whose graph is contained in $G$.
\end{lemma}
\begin{proof}
  We may assume $A_0$ is $G$-invariant by replacing $A_0$ with $[A_0]_{E_G}$.
  Apply Proposition~\ref{prop:pointaway} to find a one-ended
  Borel functions $f_i \from (A_i \setminus A_{i+1}) \to A_i$ whose graph
  is contained in $G$. Then let $f = \bigunion_i f_i$. The function $f$
  will be one-ended since any decreasing sequence in $f$ would contain a
  subsequence which is decreasing in some $f_i$.
\end{proof}

Indeed, the hypothesis of Lemma~\ref{lem:iterative_oes} is equivalent to
the existence of a Borel one-ended function whose graph is contained in
$G$; if $f$ is a
one-ended Borel function, then let $A_n = f^n[X]$.

Next, we use Lemma~\ref{lem:iterative_oes} to construct one-ended
Borel functions with conull domain in bounded degree acyclic Borel
graphs with an additional property that we call ampleness.

\begin{definition}\label{def:ample}
We say that a graph $G$ on $X$ is \define{ample} if every vertex has degree
at least $2$, and for all $x \in X$ every connected component of $G
\restriction (X \setminus \{x\})$ contains a vertex of degree at least $3$.
\end{definition}

Geometrically, an acyclic locally finite graph is ample if it contains no
isolated ends. Equivalently, an acyclic locally finite graph $G$ is ample
if it can be obtained from an acyclic graph with each vertex of degree at
least $3$ by ``subdividing'' each edge by adding some vertices of degree $2$.

\begin{lemma}\label{lem:ample}
  Suppose that $G$ is a bounded-degree acyclic Borel graph on a standard
  Borel space $X$. Suppose moreover that $G$ is ample. Let $\mu$ be a Borel
  probability measure on $X$. Then there is a $\mu$-conull Borel set $B$
  and a one-ended Borel function $f \from B \to X$ whose graph is contained
  in $G$.
\end{lemma}
\begin{proof}
  Fix $d$ bounding the degree of vertices of $G$, so $d \geq 3$.
  The heart of the construction rests in the following claim.

  \begin{claim}\label{claim:subample}
  There is a Borel subset $A \subset X$ meeting each connected component of $G$ and with $\mu(A) \leq 1 - d^{-3}$, such that $G \restriction A$ is ample.
  \end{claim}
  \begin{proof}[Proof of the claim]
    Let $X' = \{x \in X : \deg_G(x) \geq 3\}$ and define an auxiliary graph
    $G'$ on $X'$ by putting $x \G' y$ if $x \mathrel{\EG} y$ and the unique
    $G$-path from $x$ to $y$ contains no other points of $X'$.  Define a
    Borel map $\pi \from X \to X'$ selecting for each $x$ a closest element
    of $X'$ with respect to the graph metric on $G$. Let $\nu = \pi_* \mu$
    be the pushforward measure of $\mu$ on $X'$ under $\pi$, so $\nu(B) =
    \mu(\pi^{-1}(B))$ for all Borel $B$.

    Finally, let $H$ be the distance $\leq 3$ graph associated with $G'$,
    so two distinct points of $X'$ are $H$ related if they are connected by
    a $G'$ path of length at most $3$. Now $H$ has degree bounded by
    $d^3-1$, and hence by~\cite{KST}*{Proposition 4.6} a Borel coloring in
    $d^3$ colors.  Consequently, there is an $H$-independent Borel set $C'
    \subset X'$ with $\nu(C') \geq d^{-3}$.

    Define $C \subset X$ by $x \in C$ if $x \in C'$ or $x \in X \setminus
    X'$ and can be connected to a point in $C$ without using any other
    points of $X'$.  Note that $\pi^{-1}(C') \subset C$, so in particular $\mu(C) \geq d^{-3}$.  We then set $A = X \setminus C$, and check that $A$ satisfies the conclusion of the claim.

    The $G'$-independence of $C$ (in conjunction with the ampleness of $G$)
    implies that $A$ meets each $G$-component. The only thing remaining to
    check is that $G \restriction A$ is ample. Note that the only way a
    vertex $x$ in $X'$ can have ($G \restriction A$)-degree less than three
    is if it is $G'$-adjacent to an element of $C$.  So the fact that
    distinct points of $C$ have $G'$ distance at least four implies that
    $x$ has two $G'$ neighbors in $X'$ whose ($G \restriction A$)-degree remains $3$.  In particular, the degree of $x$ is two.  Moreover, if $x$ were used to witness the ampleness condition of one of its neighbors, the condition can be witnessed instead by the other neighbor.  So $G \restriction A$ is ample and the claim is proved.
  \end{proof}

  By iterating the claim, we may build a decreasing sequence $A_0 \supset
  A_1 \supset \ldots$
  of Borel sets so that $A_0 = X$, $A_{i+1}$ meets each component of
  $G \restriction A_i$, and $\mu\left(\biginters_i A_i \right)=0$. Since we
  may assume that $G$ is $\mu$-quasi-invariant, after discarding the
  $\mu$-null saturation of $\biginters_i A_i$, we can then apply
  Lemma~\ref{lem:iterative_oes}.
\end{proof}

The same idea works in the context of Baire category, even in the more
general context where $G$ is not bounded degree.

\begin{lemma}\label{lem:ample_bc}
  Suppose that $G$ is a locally finite acyclic Borel graph on a standard
  Borel space $X$. Suppose moreover that $Y$ is a Borel set and $G
  \restriction Y$ is ample. Let $\tau$ be a Polish topology compatible with
  $X$. Then there is a $G$-invariant Borel set $B \subset Y$ such that $Y
  \setminus B$ is $\tau$-meager and a
  one-ended Borel function $f \from B \to X$ whose
  graph is contained in $G$.
\end{lemma}
\begin{proof}
  The proof is very similar to
  Lemma~\ref{lem:ample} above.
  The
  statement to prove in place of the above claim is the following: $(*)$
  For any non-empty $\tau$-open set $U$ there is a Borel subset $A\subset
  Y$ meeting each connected component of $G \restriction Y$ and with $U \setminus A$
  non-meager, such that $G\restriction A$ is ample. The proof of $(*)$ is
  the same as the proof of the claim except that we choose the
  $H$-independent Borel set $C\subset X_3$ with $U\inters \pi ^{-1}(C)$
  non-meager. Then we fix a countable base $\{U_k\} _{k\in \N}$ of open
  sets for $\tau$ and, as in part \emph{1.}, we iteratively apply $(*)$ to
  build a decreasing sequence $(A_i)_{i \in \N}$ of Borel sets so that $A_0
  = X$, $A_{i+1}$ meets each component of $G \restriction A_i$, and with
  $U_i\setminus A_i$ non-meager. It follows that $U_k\setminus \biginters
  _i A_i$ is non-meager for all $k\in \N$, and therefore $\biginters _i
  A_i$ is meager. The rest of the proof is as before.
\end{proof}

Next, we show that we can reduce the problem of proving
Theorem~\ref{thm:general_one_end} to the case of ample graphs.
\begin{lemma}\label{lem:reduce_to_ample}
  Suppose $G$ is a locally finite acyclic Borel graph on a standard Borel
  space $X$ and no connected component of $G$ has $0$ or $2$ ends. Then
  there is a Borel set $B$ such that $G \restriction B$ is ample, and there
  is a one-ended Borel function $f \from (X \setminus [B]_{E_G}) \to (X
  \setminus [B]_{E_G})$ contained in the graph of $G \restriction (X
  \setminus [B]_{E_G})$.
\end{lemma}
\begin{proof}
  Let $A$ be the set of $x \in X$ such that there are disjoint
  rays  $(y_i)_{i \in \N}$ and $(z_i)_{i \in \N}$ such that $y_0$ and $z_0$
  are neighbors of $x$. Note that $A$ is Borel by K\"onig's lemma.
  Now every vertex in the induced subgraph $G \restriction A$ has degree
  at least $2$. Furthermore, every connected component of $G \restriction
  A$ has at least $3$ ends, since $A$ does not meet any connected component
  of $G$ with $1$ end, $G$ contains no connected components with $0$ or
  $2$ ends.

  Consider the set $X \setminus [A]_{E_G}$ of connected components that do
  not contain any element of $A$. This is the set of connected components
  of $G$ that each have $1$ end. Clearly $G \restriction X \setminus
  [A]_{E_G}$ has a Borel one-ended subforest on
  this set: map each $x$ to the unique neighbor $y$ such that there is an
  injective $G$-ray $(z_i)$ with $z_0 = x$ and $z_1 = y$.

  For each $i \in \N$, let $A_n$ be the set of $x \in A$ such that $x$ has
  distance at least $n$ from every vertex $y$ with $\deg_{G \restriction
  A}(y) \geq 3$, and $x$ is contained in an isolated end of $G \restriction
  A$. Here by $x$ being contained in an isolated end, we mean that there is
  an injective $(G \restriction A)$-ray $(x_i)_{i \in \N}$ such that $x =
  x_0$, and every $x_i$ has degree $2$. Note that each $A_n$ is a Borel
  set. Now applying
  Lemma~\ref{lem:iterative_oes} to the sequence $A_0 \supset A_1 \supset
  \ldots$ we can find a one-ended Borel function of $G \restriction
  [A_0]_{E_G}$.

  Let $B = A \setminus [A_0]_{E_G}$. Then clearly $G \restriction (X
  \setminus [B]_{E_G})$ has a Borel one-ended spanning subforest (since $G
  \restriction (X \setminus [A]_{E_G})$ and $G \restriction [A_0]_{E_G}$
  both do). Furthermore, $G \restriction B$ is ample, since every vertex of
  $G \restriction A$ has degree at least $2$, and $G \restriction B$ has no
  isolated ends by the definition of $A_0$.
\end{proof}

We note that Lemma~\ref{lem:reduce_to_ample} generalizes to locally
countable graphs, but where $B$ is analytic, and the one-ended subforest of
$G \restriction (X \setminus [B]_{E_G})$ will be
$\sigma(\bfSigma^1_1)$-measurable.

We are now ready to prove Theorem~\ref{thm:general_one_end}.
\begin{theorem}
  Suppose that $G$ is a locally finite acyclic Borel graph on a standard
  Borel space $X$, such that no connected component of $G$ has $0$ or $2$
  ends.
  \begin{enumerate}
   \item Let $\mu$ be a Borel probability measure on $X$. Then there is a
   $\mu$-conull Borel set $B$ and a one-ended Borel function $f \from B \to
   X$ whose graph is contained in $G$.

   \item Let $\tau$ be a compatible Polish topology on $X$. Then there is a
   $\tau$-comeager Borel set $B$ and a one-ended Borel function $f \from B
   \to X$ whose graph is contained in $G$.
   \end{enumerate}
\end{theorem}
\begin{proof}
  \emph{2} follows directly from Lemma~\ref{lem:reduce_to_ample} and then
  Lemma~\ref{lem:ample_bc} and Proposition~\ref{prop:pointaway}.

  We prove \emph{1}. By Lemma~\ref{lem:reduce_to_ample}
  we may assume that $G$ is ample. By Lemma~\ref{lem:ample} and
  Proposition~\ref{prop:pointaway} the theorem is true when
  $G$ has bounded degree.

   By
  \cite{KST}*{Proposition 4.10}, we may find a Borel edge coloring of $G$
  with $\N$ colors. Let $G_n$ be the subgraph of $G$ consisting of all
  edges assigned a color $\leq n$ so that $G_n$ is a bounded degree Borel
  graph. Let $A_n = \{x \in X : \text{$G_n \restriction [x]_{E_{G_n}}$ does not
  have $0$ or $2$ ends}\}$. Then the graph $G_n \restriction A_n$ is bounded
  degree, and so we can find a Borel one-ended subforest of each $G_n
  \restriction A_n$ modulo a nullset, by our observation about bounded
  degree graphs above.

  Thus, modulo a nullset, we can find a Borel one-ended subforest of each
  $G \restriction [A_n]_{E_G}$ via Proposition~\ref{prop:pointaway}, and
  hence a Borel one-ended subforest of $G \restriction [\bigunion_n
  A_n]_{E_G}$.

  So we need only need to construct our one-ended subforest on the graph $G \restriction (A \setminus [\bigunion_n
  A_n]_{E_G})$. Now each $G_n \restriction (X \setminus
  A_n)$ is either $0$ or $2$-ended. Thus, each $G_n \restriction (X \setminus
  A_n)$ is $\mu$-hyperfinite since connected components with $0$ ends are
  finite, and for those with $2$ ends we can apply \cite{JKL}*{Lemma 3.20}.
  Hence $G
  \restriction (A \setminus [\bigunion_n A_n]_{E_G})$ is an increasing union of
  $\mu$-hyperfinite graphs and is hence also $\mu$-hyperfinite by
  \cite{KM}*{Theorem 6.11}. Thus, by a result of Adams, \cite{JKL}*{Lemma 3.21}, there is a Borel assignment of one or two ends to each
  equivalence class of the graph $G \restriction (A \setminus [\bigunion_n
  A_n]_{E_G})$, modulo a nullset. Let $C_1$ be set where
  there is a Borel assignment of one end, and $C_2$ be the
  subset where there is a Borel assignment of two ends.

  Since $G$ is acyclic, any two ends in a connected component of $G$ are
  jointed by a unique line. Let $B_n \subset C_2$ be the set of points of
  distance at least $n$ from this distinguished line. Since $G \restriction
  C_2$ is acyclic, each point has degree at least $2$, and each connected
  component does not have $2$ ends, $B_{n+1}$ meets each connected
  component of $G \restriction B_{n}$ and hence we can apply
  Lemma~\ref{lem:iterative_oes} to the sequence $B_0
  \supset B_1$ to find a
  one-ended function contained in the graph of $G \restriction C_2$.

  On the set $C_1$, there is a Borel function $g \from C_1 \to C_1$
  generating $G \restriction C_1$ corresponding to the unique choice of
  end. In particular, for every $x,y \in C$ there are $n,m \in N$ such that
  $f^n(x) = f^m(y)$. By Lusin-Novikov uniformization~\cite{K}*{Theorem
  18.10}, we can find a Borel function $h \from C_1 \to C_1$ such that
  for every $x \in X$, $g(h(x)) = x$. Consider the set $C^* \subset C_1$ of
  connected components in $G \restriction C_1$ on which $h$ is not
  one-ended. Then the set of $x$ such that $h(g(x)) = x$ forms a unique
  bi-infinite line in $C^*$, and hence we can find a one-ended Borel
  function on $C^*$ as we did above on the set $C_2$. On $C_1 \setminus
  C^*$, the function $h$ is one-ended, so we are done.
\end{proof}

We note here that \emph{(1)} in Theorem~\ref{thm:general_one_end} can be
generalized to locally countable graphs. Several of the functions and sets
that we have used in our argument will be analytic and
$\sigma(\bfSigma^1_1)$-measurable, but these will become Borel after
discarding a nullset. The only other modification we need to make is that in the proof above,
locally countable acyclic
graphs with $0$ ends are not necessarily finite, but they are smooth by
\cite{M09}*{Theorem A}, and hence hyperfinite.

In recent work joint with Damien Gaboriau, the authors have extended
Theorem~\ref{thm:general_one_end} to characterize exactly when a (not
necessarily acyclic) measure preserving locally finite graph $G$ has a
one-ended spanning subforest.

Now we use the ability to find one-ended functions inside a graph to help
definably color the graph.

The following proposition is a trivial modification of \cite{KST}*{Proposition 4.6}. It is proved by partitioning the set $B$ into countably many
$G$-independent Borel sets
$A_0, A_1, \ldots$ and then coloring each vertex the least color in $L(x)$
(with respect to a Borel linear ordering of $Y$) not
already used by one of its neighbors.
\begin{prop}\label{prop:KST_dlc}
  Suppose $G$ is a locally finite Borel graph on a standard Borel space $X$
  and $B \subset X$ is Borel. Then if $Y$ is a Polish space and $L \from B
  \to [Y]^{< \infty}$ is a Borel function such that for every $x \in B$,
  $d_{G \restriction B}(x) < |L(x)|$. Then $G \restriction B$ has a Borel
  coloring from the lists $L$.
\end{prop}

We now have the following lemma which is essentially identical to \cite{CK}*{Lemma 2.18}.
\begin{lemma}\label{lem:partial_dlc}
  Suppose that $G$ is a locally finite Borel graph on a standard Borel
  space $X$, $B$ is a Borel subset of $X$ and $f
  \from B \to X$ is a one-ended Borel function whose graph is contained in
  $G$. If $Y$ is a Polish space and $L \from B \to [Y]^{< \infty}$ is a
  Borel function such that $L(x) \geq \deg_G(x)$ for every $x \in B$, then $G
  \restriction B$ has a Borel coloring from the lists $L$.
\end{lemma}
\begin{proof}
Let $B_i = B \inters (f^i[B] \setminus f^{i+1}[B])$ so that $B_i$ consists
of the points in $B$ that have rank $i$ in the graph generated by $f$.
Note that $B$ is the disjoint union of $B_0, B_1, \ldots$. We will define
a coloring $c$ of $G \restriction B$ from the lists $L$. Let $L_0 = L$.
Now iteratively apply Proposition~\ref{prop:KST_dlc} to color $G
\restriction B_i$ from the lists $L_i$, and then define $L_{i+1}(x) = L_i(x)
\setminus \{c(y) : y \in N(x) \land (\exists j < i ) y \in B_j\}$.
Note that since each $x \in B_i$ has at least one
neighbor $f(x)$ not in $B_0, \ldots, B_{i-1}$, that $d_{G \restriction
B_i}(x) < L_i(x)$.
\end{proof}

\begin{corollary}\label{cor:oneend}
  Suppose that $G$ is a locally finite Borel graph on a standard Borel
  space $X$ and there exists a one-ended Borel function $f \from X \to X$
  whose graph is contained in $G$. Then $G$ is Borel degree-list-colorable.
\end{corollary}

We note that this implies that for every finite $d$ there is an acyclic
Borel graph $G$ on $X$ of degree $d$ such that there is no one-ended Borel
function $f \from X \to X$ whose graph is contained in $G$. This is because
by~\cite{Ma}*{Theorem 1.3} for every finite $d$, there is an acyclic
Borel graph of degree $d$ with no Borel $d$-coloring.

\section{A proof of the measurable Brooks's theorem}


We are now ready to prove Theorem~\ref{thm:Borel_Gallai} from the
introduction.

\begin{thm}\label{no_gallai}
Suppose that $G$ is a locally finite Borel graph on a standard Borel space
$X$ and let $B$ be the set of vertices contained in connected components of
$G$ that are not Gallai trees. Then $G \restriction B$ is Borel degree-list-colorable.
\end{thm}
\begin{proof}
  Let $[E_G]^{< \infty} \subset [X]^{< \infty}$ be the finite subsets $S$ of
  $X$ that are contained in a single connected component of $G$. Let $G_I$
  be the intersection graph on $[E_G]^{< \infty}$ so $R \mathrel{G_I}
  S$ if $R \inters S \neq \emptyset$. Then $G_I$ has a Borel $\N$-coloring
  $c_I$ (see \cite{KM}*{Lemma 7.3} and \cite{CM2}*{Proposition 2}).

  If $S \subset X$ is a finite set, its boundary in $G$ is $\partial S =
  \{y \notin S : \exists x \in S (x \mathrel{G} y)\}$.
  Let $A \subset [E_G]^{< \infty}$ be collection of finite sets $S$ such
  that each connected component of $G \restriction S$ is not a Gallai tree.
  Let $A' \subset A$ be the set of $S \in A$ such that $c_I(S \union
  \partial S) \leq c_I(R \union \partial R)$ for all $R \in A$ in the
  same $G$-component as $S$. Let $B' = \bigunion A'$, so $B' \subset B$,
  and each connected component of $G \restriction B'$ is finite and not a
  Gallai tree, and $B'$ meets each connected component of $G \restriction
  B$.

  Let $L \from X \to [Y]^{< \infty}$ be an assignment of lists to each
  element of $x$ so that $\deg_G(x) = |L(x)|$.
  By Proposition~\ref{prop:pointaway}, we can find a one-ended
  function $f \from B \setminus B' \to B$, and by
  Lemma~\ref{lem:partial_dlc} we can find a coloring $c$ of $G \restriction (B
  \setminus B')$ from the lists $L$. Now let $L' \from B' \to [Y]^{<
  \infty}$ be defined by $L'(x) = L(x) \setminus \{c(y) : y \mathrel{G} x
  \mathand y \in B \setminus B'\}$. To finish, there is at least one coloring of each
  connected component of
  $G \restriction B'$ from the lists $L'$. Hence, by Lusin-Novikov
  uniformization~\cite{K}*{Theorem 8.10} we can extend $c$ to a coloring of
  $G \restriction B$ from the lists $L$.
\end{proof}

We are now ready to prove a version of Theorem~\ref{thm:mb} for list
colorings.

\begin{thm}\label{thm:mb_listcoloring}
Suppose that $G$ is a locally finite Borel graph on a standard Borel space
$X$ and $G$ contains no connected components that are finite Gallai trees,
and no infinite connected components that are $2$-ended Gallai trees.
\begin{enumerate}
\item Let $\mu$ be any Borel probability measure on $X$. Then there is a
$\mu$-conull $G$-invariant Borel set $B$ so that $G \restriction B$ is
Borel
degree-list-colorable.
\item Let $\tau$ be any Polish topology compatible with the Borel structure
on $X$. Then there is a $G$-invariant comeager Borel set $B$ so that $G
\restriction B$ is Borel degree-list-colorable.
\end{enumerate}
\end{thm}
\begin{proof}
  The theorem follow by combining Theorem~\ref{no_gallai},
  Theorem~\ref{thm:general_one_end}, and Corollary~\ref{cor:oneend}.

  Let
  $A$ be the set of vertices contained in connected components of $G$ that
  are not Gallai trees. Then $G \restriction A$ is Borel
  degree-list-colorable by Theorem~\ref{no_gallai}. Hence, we may as well
  assume that every connected component of $G$ is an infinite Gallai tree.

  Now let $Y \subset [X]^{< \infty}$ be the Borel set of blocks of $G$, and
  consider the intersection graph $G_I \restriction Y$ on blocks so that
  two distinct blocks $R, S \in Y$ are adjacent if $R \inters S \neq
  \emptyset$. Since these blocks are maximal biconnected components of $G$,
  there cannot be any cycles in $G_I$, since such a cycle would imply its
  constituent blocks were not maximal biconnected components of $G$.
  Similarly, any two blocks intersect at a unique vertex.
  Finally, since no connected components of $G$ are $2$-ended Gallai
  trees, no connected components of $G_I \restriction Y$ have $0$ or $2$
  ends.

  Thus, by Theorem~\ref{thm:general_one_end} we can find a Borel one-ended
  subforest of $G_I \restriction Y$ modulo a null or meager set induced by
  a function $f$. The function $f$ then ``lifts'' to a one-ended function
  $\hat{f}$ contained in
  $G$ as follows. Fix a Borel linear orderings $<_X$ of $X$ and $<_Y$ of
  $Y$. Now given a vertex $x$, since each connected component of $G$ is a
  locally finite infinite Gallai tree, $x$ is contained at least one and at
  most finitely many blocks of $G$. Let $g(x)$ be the $<_Y$-least block
  containing $x$. Now define a Borel function $g'(x)$ by letting $g'(x) =
  g(x)$ if $x$ is not contained in $f(g(x))$, and $g'(x) = f(g(x))$
  otherwise. Hence, $g'(x)$ maps each vertex $x$ to a block containing $x$
  so that $x$ is not in $f(g'(x))$. Now let the function $\hat{f}$ map each
  $x$ to the next vertex along the $<_X$-lex least path from $x$ to an
  element of $f(g'(x))$.

  We then finish the proof of the theorem
  by applying Corollary~\ref{cor:oneend} to $\hat{f}$.
\end{proof}

We note that one application of the above theorem is a new way of constructing
antimatchings (see \cite{Ma}).
Recall that an \define{antimatching} of a graph $G$ on a set
$X$ is a Borel function $f \from X \to X$ contained in the graph of $G$
such that for every $x$, we have $f(f(x)) \neq x$. If $G$ is a
locally finite graph, then we can map each $x$ to the set $L(x)$ of edges
in $G$ incident to $x$. Then a coloring $c$ of $G$ from the lists $L$ can
be used to define an antimatching, by letting $f(x)$ be the unique neighbor
$y$ of $x$ in the edge $c(x) = \{x,y\}$.

We can now prove Theorem~\ref{thm:mb} from the introduction:
\begin{proof}[Proof of Theorem~\ref{thm:mb}:]
  Suppose $G$ is a Borel graph of bounded degree at most $d$, and $G$ does
  not contain a complete graph on $d+1$ vertices. Let $A$ be the set of
  vertices of degree strictly less than $d$. We begin by $d$-coloring the
  connected components $[A]_{E_G}$ contain an element of $A$. Our idea is
  to color some element of $A'$ ``last.''

  To begin, let $A' \subset A$ be $G$-independent Borel set
  that meets every connected component of $G \restriction
  [A]_{E_G}$. Such an $A'$ exist by taking a Borel $(d+1)$-coloring of $G$
  by \cite{KST}*{Proposition 4.6} and then letting $A'$ be the elements of
  $A$ assigned the least color among all elements
  of $A'$ in the same connected component. Now apply
  Proposition~\ref{prop:pointaway}, and Lemma~\ref{lem:partial_dlc}
  to obtain a Borel
  $d$-coloring of $G \restriction [A]_{E_G} \setminus A'$, and then color
  each element of $A'$ the least color not already used by one of its
  neighbors.

  To finish, we need to color the remainder $G \restriction (X \setminus
  A)$ and so it suffices to show that Theorem~\ref{thm:mb} is true for
  $d$-regular graph. But this follows from
  Theorem~\ref{thm:mb_listcoloring} since the only finite Gallai trees that
  are $d$-regular are complete graphs on $d+1$ vertices, and the only
  infinite regular two-ended Gallai trees are bi-infinite lines.
\end{proof}

We briefly discuss an alternate way of proving Theorem~\ref{thm:mb}. The case
$d = 3$ of the theorem is fairly easy to analyze directly. One can then
reduce to the case $d = 3$ by
iteratively
removing maximal independent sets meeting every $d$-clique using the
following proposition.

\begin{prop}\label{lemma:cliquebusting}
  Suppose $G$ is a Borel graph on a standard Borel space $X$ of finite
  bounded degree $\leq d$, where $d \geq 3$. Suppose further that $G$
  contains no cliques on $d+1$ vertices.
\begin{enumerate}
  \item Let $\mu$ be any Borel probability measure on $X$. Then there is a
  $\mu$-measurable maximal independent set $A \subset X$ that meets every
  $d$-clique contained in $G$.

  \item Let $\tau$ be any Polish topology compatible with the Borel structure
  on $X$. Then there is a
  Baire measurable maximal independent set $A \subset X$ that meets every
  $d$-clique contained in $G$.
\end{enumerate}
\end{prop}

Of course, this proposition follows from Theorem~\ref{thm:mb} by extending one
of the colors in a $d$-coloring (which must meet every $d$-clique) to a
maximal independent set. However, it is also simple to prove this
proposition directly. Let $Y \subset X$ be the vertices that are contained
in a unique $d$-clique. Then let $E$ and $F$ be the relations on $Y$ where
$x \E y$ if the unique $d$-cliques containing $x$ and $y$ are equal and $x
\mathrel{F} y$ if $x = y$ or $x$ and $y$ are adjacent in $G$ and are not
$E$-related. One can then use \cite{Ma}*{Lemma 4.4.1} to find a $\mu$-measurable or Baire
measurable set $A$ meeting every $E$-class in exactly one point and every
$F$-class in at most one point. From here, extending $A$ to the desired set
is straightforward.

\section{Applications to group actions}\label{sec:applications}

We consider now (almost everywhere) free, measure-preserving actions of a finitely generated group $\Gamma$ on a standard probability space $(X,\mu)$.  Denote by $\freeactions(\Gamma,X,\mu)$ the set of such actions.  With each $a \in \freeactions(\Gamma,X,\mu)$ and finite, symmetric generating set $S$ of $\Gamma$ not containing the identity we may associate a graph $G(S,a)$ on $X$ by declaring $x$ and $y$ adjacent if there exists $s \in S$ with $s \cdot x = y$.  Freeness of the action implies that almost every connected component of $G(S,a)$ is isomorphic to the Cayley graph $\Cay(\Gamma, S)$.

In \cite{CKT}*{Theorem 6.1} it is shown that for finitely generated infinite groups $\Gamma$, any $a \in \freeactions(\Gamma,X,\mu)$ is weakly equivalent to some $b \in \freeactions(\Gamma,X,\mu)$ whose associated graph $G(S,b)$ is measure-theoretically $|S|$-colorable.  Theorem \ref{thm:mb} eliminates the need to pass to a weakly equivalent action for almost all groups.

\begin{corollary}
  Suppose that $\Gamma$ is an infinite group with finite, symmetric generating set $S$ such that $|S| \geq 3$.  Then for any $a \in \freeactions(\Gamma,X,\mu)$ the graph $G(S,a)$ admits a Borel $|S|$-coloring on a conull set.
\end{corollary}

\begin{remark}
  The only infinite groups with symmetric generating sets $S$ satisfying $|S| < 3$ are $\Z$ with $S = \{\pm 1\}$ and $(\Z/2\Z)*(\Z/2\Z) = \langle a, b \mid a^2 = b^2 = \id \rangle$ with $S = \{a,b\}$.  Indeed, no graph associated with a free mixing action of either group admits a Borel $2$-coloring on a conull set.
\end{remark}

Finally, the methods of section \ref{section:subforests} may be used in
conjunction with some techniques from probability to improve known bounds
on the colorings of Cayley graphs attainable by factors of IID.
We
consider the \define{Bernoulli shift} action of a countable group $\Gamma$
on the space $[0,1]^\Gamma$ equipped with product Lebesgue measure $\mu$,
where $\gamma \cdot x(\delta) = x(\gamma^{-1}\delta)$.  Denote by
$G(\Gamma,S)$ the graph associated with the Bernoulli shift and generating
set $S$.
For convenience we sometimes work instead with the shift action of $\Gamma$ on $[0,1]^E$, where $E$ is the edge set of the (right) Cayley graph $\Cay(\Gamma,S)$ (and as usual $\Gamma$ acts by left translation on the Cayley graph).  We denote the corresponding graph on $[0,1]^E$ by $G'(\Gamma,S)$.  Since the shift action on $[0,1]^E$ is measure-theoretically isomorphic to the Bernoulli shift on $[0,1]^\Gamma$, we lose nothing by working with $G'(\Gamma,S)$ rather than $G(\Gamma,S)$.

We may use each $x \in [0,1]^E$ to label the edges of its connected
component in $G'(\Gamma,S)$, assigning $(\gamma \cdot x, s\gamma \cdot x)$
the label $x(\gamma^{-1}, \gamma^{-1} s^{-1})$.  The structure of the
action ensures that this labeling is independent of the particular choice
of $x$, and in particular this labeling is a Borel function from
$G'(\Gamma,S)$ to $[0,1]$.  Following \cite{LPS} we obtain the \define{wired minimal spanning forest}, $\wmsf(G'(\Gamma,S))$, by deleting those edges from $G'(\Gamma,S)$ which receive a label which is maximal in some simple cycle or bi-infinite path.  By construction, $\wmsf(G'(\Gamma,S))$ is acyclic.

\begin{thm}[Lyons-Peres-Schramm]\label{theorem:LPS}
  Suppose that $\Gamma$ is a nonamenable group with finite symmetric generating set $S$, and consider the graph $G'(\Gamma,S)$ defined above.  There is a conull, $G'(\Gamma,S)$-invariant Borel set $B \subset [0,1]^E$ on which each connected component of $\wmsf(G'(\Gamma,S))$ has one end.
\end{thm}

\begin{proof}
  See \cite{LPS}*{Theorem 3.12}, which says $\wmsf(G'(\Gamma,S))$ is almost surely
  one-ended provided the Cayley graph of $\Gamma$ has no infinite clusters
  at critical percolation. This holds for nonamenable Cayley graphs by
  \cite{BLPS:nonamenable}*{Theorem 1.1}.
\end{proof}

Let $\Aut_{\Gamma,S}$ be the automorphism group of the Cayley graph
$\Cay(\Gamma,S)$. Given a group $\Gamma$ with generating set $S$ and a
natural number $k$, we may view the space $\Col(\Gamma,S,k)$ of
$k$-colorings of the (right) Cayley graph $\Cay(\Gamma,S)$ as a closed
(thus Polish) subset of $k^\Gamma$.  The action of $\Gamma$ by left
translations on $\Cay(\Gamma,S)$ induces an action on $\Col(\Gamma,S,k)$.
An \define{automorphism-invariant random $k$-coloring} of $\Cay(\Gamma,S)$
is a Borel probability measure on $\Col(\Gamma,S,k)$ invariant under this
$\Aut_{\Gamma,S}$ action. Such a random $k$-coloring is a \define{factor of
IID} if
it is a factor of the Bernoulli shift of $\Aut_{\Gamma,S}$ on
$[0,1]^{\Gamma}$. That is, letting $\lambda$ be Lebesgue measure on
$[0,1]$, a random $k$-coloring $\nu$ is a factor of IID if there is a
$\mu$-measurable equivariant function $f \from [0,1]^\Gamma \to k^\Gamma$
such that $\nu$ is the pushforward of the product measure $\mu^\Gamma$
under $f$.

In Section 5 of \cite{LN} it is asked for which $k$ can
automorphism-invariant random $k$-colorings of Cayley graphs be attained as
IID factors (see also \cite{AL}*{Question 10.5}). In \cite{CKT}*{Corollary 6.4}
translation-invariant random $d$-colorings of Cayley graphs are
constructed, where as usual $d$ is the degree of the graph, but this involves passing to actions weakly equivalent to the
Bernoulli shift (or alternatively taking weak limits of IID factors). We
can now strengthen this result, giving $d$-colorings as IID factors except in the cases
$\Z$ and $(\Z/2\Z) * (\Z/2\Z)$ where there is the usual ergodic-theoretic
obstruction.

\begin{corollary}\label{cor:randdcolor}
  Suppose that $\Gamma$ is a countable group not isomorphic to $\Z$ or
  $(\Z/2\Z) * (\Z/2Z)$, and suppose that $S$ is a finite symmetric
  generating set for $\Gamma$ with $|S| = d$.  Then there is an
  automorphism-invariant random $d$-coloring of $\Cay(\Gamma,S)$ which is an IID factor.
\end{corollary}
\begin{proof}
  In the case that $\Gamma$ is amenable, it has finitely many ends, and so
  we can apply \cite{CKT}*{Theorem 6.7}. Otherwise, $\Gamma$ is
  nonamenable and we can apply Corollary \ref{cor:oneend} to obtain from $\wmsf(G'(\Gamma,S))$ a Borel $d$-coloring $c \from B \to d$ of the restriction of $G(\Gamma,S)$ to the conull set $B \subset [0,1]^\Gamma$ on which $\wmsf(G'(\Gamma,S))$ has one end.  Define $\pi \from B \to \Col(\Gamma,S,d)$ by $(\pi(x))(\gamma) = c(\gamma^{-1} \cdot x)$.  Then $\pi_* \mu$ is a automorphism-invariant random $d$-coloring which is a factor of IID by construction, where as usual $\pi_* \mu(A) = \mu(\pi^{-1}(A))$.
\end{proof}

\begin{remark}
  Russ Lyons (private communication) points out that this method of proof
  using spanning forests works for finitely generated groups of
  more than linear growth by using instead the wired \define{uniform} spanning
  forest (WUSF); see Section 10 of \cite{BLPS:uniform}.  The realization of
  the WUSF as a factor of IID follows from Wilson's algorithm rooted at
  infinity (see \cite{GL}*{Proof of Proposition 9}) in the transient case and Pemantle's strong F{\o}lner independence \cite{Pemantle} in the amenable case.
\end{remark}

\section{The case $d=2$}\label{sec:d=2}

In this section, we prove Theorem~\ref{thm:mb2}, giving a measurable
analogue of Brooks's theorem for the case $d = 2$.

Given a graph $G$ on $X$, let the equivalence relation $\evenG$ be the
equivalence relation on $X$ where $x \evenG y$ if $x$ and $y$ are connected
by a path of even length in $G$. Then in the case where $X$ is finite, we
can rephrase the existence of an odd cycle in the following way: there is a
nonempty $G$-invariant subset $A$ of $X$ such that every nonempty
$\evenG$-invariant subset of $A$ is $G$-invariant.

Now, in the measurable context, even without the presence of odd cycles,
there are Borel graphs $G$ and measures $\mu$ for which every $\evenG$-invariant Borel set differs by a
nullset from a Borel $G$-invariant set.
For example, the Borel graph $G_S = \{ (x,y)\in \T ^2 : S(x) = y \mathor S(y)
= x\}$ induced by an irrational rotation $S:\T \rightarrow \T$ of the unit
circle is $2$-regular and acyclic, and since $S^2$ is ergodic with respect
to Lebesgue measure, every non-null $\evenG$-invariant Borel set
is Lebesgue conull. It follows that $G_S$ does not admit a $\mu _\T$-a.e.\ Borel
$2$-coloring, as the color sets in a measurable $2$-coloring would have to
be disjoint, $\evenG$-invariant, and
non-null since $G_S$ is induced by a measure preserving transformation and
is hence quasi-invariant. Likewise, there is no Baire measurable $2$-coloring of $G_S$
with respect to the usual topology on $\T$ since every non-meager
$\evenG$-invariant Borel set of vertices in $G_S$ is comeager.

If we regard the phenomenon described above as generalization of
possessing on odd cycle, then we have the following generalization of
Brooks's theorem in the case $d = 2$:

\begin{thm}
  Suppose $G$ is a Borel graph on a standard Borel space $X$ with vertex
  degree bounded by $d = 2$ such that $G$ contains no odd cycles. Let
  $\evenG$ be the equivalence relation on $X$ where $x \evenG y$ if $x$ and
  $y$ are connected by a path of even length in $G$.
  \begin{enumerate}
  \item Let $\mu$ be a $G$-quasi-invariant Borel probability measure on
  $X$. Then $G$ admits a $\mu$-measurable $2$-coloring if and only if there does not
  exist a non-null $G$-invariant Borel set $A$ such that every
  $\evenG$-invariant Borel subset of $A$ differs from a $G$-invariant set by a
  nullset.

  \item Let $\tau$ be a $G$-quasi-invariant Polish topology
  compatible with the Borel structure on $X$. Then $G$ admits a
  Baire measurable $2$-coloring if and only if there does not exist a non-meager
  $G$-invariant Borel set $A$ so that every $\evenG$-invariant Borel subset
  of $A$ differs from a $G$-invariant set by a meager set.
  \end{enumerate}
\end{thm}
\begin{proof}
We prove just part 1, since the proof of 2 is similar. Assume first that $G$
admits a $\mu$-measurable $2$-coloring with colors sets $C_0$ and $C_1$.
Now $C_0$ and $C_1$ must both be non-null, since $\mu$ is
$G$-quasi-invariant. However, if $A$ is a non-null $G$-invariant Borel set,
then $A \inters C_0$ is $\evenG$-invariant, however $A \inters C_0$ cannot
differ from a $G$-invariant Borel set by a nullset since $\mu$ is
$G$-quasi-invariant.

For the converse, assume that for every $\mu$-measurable non-null
$G$-invariant set $A$ we can find a $\mu$-measurable $\evenG$-invariant
$C \subset A$ which is not within a null-set of being $G$-invariant.
Then the sets $C_0 = \{ x \in C : [x]_{\EG} \setminus C
\neq \emptyset \}$ and $C_1 = [C_0]_{\EG}\setminus C_0$ are non-null, and
they determine a $\mu$-measurable $2$-coloring of $G\restriction
[C_0]_{\EG}$. We may continue this process on $X\setminus [C_0]_{\EG}$,
and by measure theoretic exhaustion we can obtain a $\mu$-measurable
$2$-coloring $c:Y \rightarrow \{ 0, 1 \}$ of $G\restriction Y$ for some
$G$-invariant conull $Y\subset X$.
\end{proof}

\begin{appendices}
\section{Borel vs measurable colorings}\label{sec:Borel_vs_measurable}

Let $(\Z/2\Z)^{*d}$ be the $d$-fold free product of the group $\Z/2Z$. This
group acts via the left shift action on the standard Borel space
$\N^{(\Z/2\Z)^{*d}}$. Let $X = \{x \in
\N^{(\Z/2\Z)^{*d}} : \text{ $\gamma \cdot x \neq x$ for all nonidentity
$\gamma \in (\Z/2\Z)^{*d}$}\}$ be the free part of this action. Let
$G((\Z/2\Z)^{*d},\N)$ be the Borel graph on $X$ where $x$ is adjacent to
$y$ if there is a generator $\gamma$ of $(\Z/2\Z)^{*d}$ such that $\gamma
\cdot x = y$ or $\gamma \cdot y = x$. Note this graph is acyclic and
$d$-regular. As discussed in the introduction,
Marks~\cite{Ma}*{Theorem 1.2} has shown that this graph has no Borel
$d$-coloring. However, our Theorem~\ref{thm:mb} shows that for every $d
\geq 3$, there is a $\mu$-measurable and Baire measurable $d$-coloring of
$G((\Z/2\Z)^{*d},\N)$ with respect to any Borel probability measure or
compatible Polish topology on $X$.

Hence, for all finite $d \geq 3$, for Borel graphs $G$,
admitting $\mu$-measurable $d$-coloring with respect to every Borel
probability measure is a strictly weaker notion than admitting a Borel
coloring, as is admitting a Baire measurable $d$-coloring with respect to
every compatible Polish topology, as witnessed by these explicit graphs given
above. In this appendix, we show that for $2$-colorings these notions
are the same, even without any degree assumptions on $G$.

\begin{proposition}\label{prop:2color}
Let $G$ be a locally countable Borel graph on a standard Borel space $X$. Then the following are equivalent:
\begin{enumerate}
\item $G$ admits a Borel $2$-coloring.
\item For every Borel probability measure $\mu$ on $X$, $G$ admits a $\mu$-measurable $2$-coloring.
\item For every compatible Polish topology $\tau$ on $X$, $G$ admits a Baire measurable $2$-coloring.
\end{enumerate}
\end{proposition}

\begin{proof}
This is actually a Corollary of a more general unpublished result of
Louveau (see \cite{M12}*{Theorem 15}). We sketch Louveau's
argument in this special case, which uses the 
$G_0$-dichotomy \cite{KST}*{Theorem 6.6}. We will
prove the equivalence of \emph{1.}\ and \emph{2.}, since the proof of the
equivalence of \emph{1.}\ and \emph{3.}\ is similar. It suffices to show
that if $G$ admits no Borel $2$-coloring then $G$ admits no
$\mu$-measurable $2$-coloring for some Borel probability measure $\mu$ on
$X$. If $G$ contains an odd cycle then it cannot be $2$-colored at all, so
we may assume that $G$ contains no odd cycles. Let $G^{\text{odd}} = \{
(x,y) \in X^2 :\, \text{d}_G (x,y) \text{ is odd} \}$, where
$\text{d}_G : X^2 \rightarrow \N \union \{ \infty \}$ denotes the graph
distance in $G$. Then $G^{\text{odd}}$ admits no Borel $\N$-coloring.
$\big($Otherwise, by \cite{KST}*{Proposition 4.2} there is a maximal $G^{\text{odd}}$-independent set $A\subset X$ which is Borel, and since $G$ contains no odd cycles the set $X\setminus A$ is $G^{\text{odd}}$-independent as well, which contradicts that $G$ admits no Borel $2$-coloring.\big)
It follows from \cite{KST}*{Theorem 6.6} that there is an injective Borel
homomorphism $f:2^\N \rightarrow X$ from the graph $G_0$ to
$G^{\text{odd}}$. Then $f$ is a homomorphism from $G_0^{{\text{odd}}} = \{
(u,v)\in (2^\N )^2 :\, \text{d}_{G_0}(u,v)\text{ is odd} \} = \{
(u,v)\in 2^\N  :\, u\text{ and }v\text{ differ on an odd number of
coordinates}\}$ to $G^{\text{odd}}$. Let $\nu$ denote the uniform product
measure on $2^\N$. Then every Borel $G_0^{\text{odd}}$-independent set is
$\nu$-null (see \cite{CK}*{Example 3.7}), hence every Borel
$G^{\text{odd}}$-independent set is $f_*\nu$-null. Fix by the Feldman-Moore
theorem \cite{KM}*{Theorem 1.3}
a sequence $(T_i)_{i \in \N}$ of Borel automorphisms such that
$\EG = \bigunion_i \graph(T_i)$ and let $\mu = \sum _i 2^{-i}(T_i)_*f_*\nu$, so that $\mu$ is a $G$-quasi-invariant probability measure with the same $G$-invariant null sets as $f_*\nu$. Suppose toward a contradiction that there is a $\mu$-measurable $2$-coloring of $G$. Then there is a Borel $2$-coloring $c:B\rightarrow \{ 0, 1\}$ of $G\restriction B$ for some $G$-invariant $\mu$-conull Borel subset $B \subset X$. Since $B$ is invariant, $c$ is a $2$-coloring of $G^{\text{odd}}\restriction B$, and since $B$ is $\mu$-conull it is $f_*\nu$-conull, so either $c^{-1}(0)$ or $c^{-1}(1)$ is a Borel $G^{\text{odd}}$-independent set with positive $f_*\nu$-measure, a contradiction.
\end{proof}

\end{appendices}

\end{document}